\documentclass[letter]{amsart}
\usepackage{amssymb}
\usepackage[bookmarks, bookmarksdepth=2, colorlinks=true, linkcolor=blue, citecolor=blue, urlcolor=blue]{hyperref}
\usepackage{color}
\usepackage{cleveref}
\usepackage{algorithm}
\newtheorem{theorem}{\bf Theorem}
\newtheorem{proposition}[theorem]{\bf Proposition}

\newtheorem{question}[theorem]{\bf Question}

\theoremstyle{definition}
\newtheorem{definition}[theorem]{\bf Definition}
\newtheorem{remark}[theorem]{\bf Remark}

\newcommand \N{{\mathbb N}}
\newcommand \Z{{\mathbb Z}}
\newcommand \Q{{\mathbb Q}}
\newcommand \R{{\mathbb R}}

\newcommand \A{{\mathbb A}}
\newcommand \PP{{\mathbb P}}
\newcommand \B{{\mathcal B}}

\DeclareMathOperator{\codim}{codim}
\DeclareMathOperator{\OO}{O}
\DeclareMathOperator{\SO}{SO}

\newcommand{\arxiv}[1]{\href{http://arxiv.org/abs/#1}{{\tt arXiv:#1}}}

\begin{document}
\title{\bf{The 4 $\times$ 4 orthostochastic variety}}

\author{Justin Chen}
\address{School of Mathematics, Georgia Institute of Technology,
Atlanta, Georgia, 30332 U.S.A.}
\email{justin.chen@math.gatech.edu}

\author{Papri Dey}
\address{Department of Mathematics, University of Missouri, Columbia, Missouri, 65211 U.S.A.}
\email{pdbdn@missouri.edu}



\begin{abstract}
Orthostochastic matrices are the entrywise squares of orthogonal matrices, and naturally arise in various contexts, including notably definite symmetric determinantal representations of real polynomials. However, defining equations for the real variety were previously known only for $3 \times 3$ matrices. We study the real variety of $4 \times 4$ orthostochastic matrices, and find a minimal defining set of equations consisting of 6 quintics and 3 octics. The techniques used here involve a wide range of both symbolic and computational methods, in computer algebra and numerical algebraic geometry.
\end{abstract}

\maketitle

\section{Introduction}
A real square matrix is called \textit{orthostochastic} if it is the entrywise square of an orthogonal matrix. Explicitly, $A = (a_{ij}) \in \R^{n \times n}$ is orthostochastic if there exists an orthogonal matrix $V = (v_{ij})$ with $a_{ij} = v_{ij}^2$ for all $i, j = 1, \ldots, n$. It follows immediately that an orthostochastic matrix is doubly stochastic (i.e. has nonnegative entries and all row and column sums equal to 1), as all rows and columns of an orthogonal matrix are unit vectors. 

As an interesting and special class of doubly stochastic matrices, orthostochastic (and their unitary generalization, unistochastic) matrices arise naturally in a number of contexts, including spectral theory \cite{NTU93, Tho69}, convex analysis \cite{AP79, Hor54},
and physics \cite{BEK05, CD08}. More recently -- and of interest in algebraic geometry -- it has been shown that orthostochastic matrices are deeply connected to definite symmetric determinantal representations of real polynomials. Indeed, for hyperbolic plane curves, every monic symmetric determinantal representation arises from certain associated orthostochastic matrices, which yields an effective algorithm \cite{CD19, Dey} for computing such representations for cubic curves. 

It is thus of interest to find intrinsic characterizations of the set of orthostochastic matrices. One approach to this is: by definition, the set of orthostochastic matrices is the image of an algebraic variety (the real orthogonal group) under a polynomial map (coordinate-wise squaring) -- thus the Zariski closure of the image is a real algebraic variety. Our goal then is to find equations for this Zariski closure, which we refer to as \textit{finding equations for the orthostochastic variety}. The equations of the $3 \times 3$ orthostochastic variety are known, which made the computation of determinantal representations of cubic curves possible, but for matrices of size $\ge 4$ no set of equations which define the orthostochastic variety were previously known. In view of this, our main result is:

\begin{theorem} \label{mainThm}
The $4 \times 4$ orthostochastic variety is defined set-theoretically by 6 quintics and 3 octics, which are known explicitly and defined over $\Z$.
\end{theorem}

We outline the remainder of the article: in Section 2 we formalize the set-up and notation, and review some basic invariants of the varieties under consideration. In Section 3 we give a procedure for obtaining a naive set of equations which always cut out a superset of the orthostochastic variety, and see how dimension counts give a simple proof of equality in the case $n = 3$. In Section 4 we consider the case $n = 4$, and detail the techniques used to find the equations in \Cref{mainThm}. Section 5 concludes with some remarks and remaining questions.

\section{Setup and basic invariants}
We begin by setting up some notation which will be used in the remainder of the article. We reserve $n$ to denote the size of a square matrix, and all matrices are considered to have real entries. For $n \in \N$, let $\OO(n)$ (resp. $\SO(n)$) denote the group of $n \times n$ orthogonal (resp. special orthogonal) matrices, which is a real algebraic variety. 

Next, the set of doubly stochastic $n \times n$ matrices is defined by the linear conditions $\sum_i a_{ij} = \sum_j a_{ij} = 1$, $a_{ij} \ge 0$, for all $i, j = 1, \ldots, n$. This can be interpreted as saying that an $n \times n$ doubly stochastic matrix is uniquely determined by any $(n-1) \times (n-1)$ submatrix obtained by deleting a row and column. Moreover, recall that the set of $n \times n$ doubly stochastic matrices equals the convex hull of all $n \times n$ permutation matrices, which is the so-called Birkhoff polytope $\B_n$, of dimension $(n-1)^2$. 

We now introduce the varieties in question. Identifying the space of $n \times n$ real matrices with $n^2$-dimensional affine space $\A^{n^2}_{\R}$, we have the coordinate-wise squaring map
\begin{align*}
\phi : \A^{n^2} \to \A^{n^2} \\ (a_{ij}) \mapsto (a_{ij}^2)    
\end{align*}
Restriction to $\OO(n)$ gives a map $\OO(n) \to \A^{n^2}$, whose image lies in $\B_n$ (note that the image of $\phi$ is contained in the non-negative orthant $\A^{n^2}_{\ge 0}$). Next, the coordinate projection $\pi : \A^{n^2} \to \A^{(n-1)^2}$, given by projecting onto the upper-left $(n-1) \times (n-1)$ submatrix, is injective on $\B_n$, so we may compose with $\pi$ to obtain $\pi \circ \phi \big|_{\OO(n)} : \OO(n) \to \A^{(n-1)^2}$. Finally, for both theoretical and practical reasons it is convenient to work in projective space, so taking the projective closure of the image yields the map
\[
\varphi : \OO(n) \to \PP^{(n-1)^2}_{\R}
\]
We set $Z_n := \overline{\varphi(\OO(n))}$, the Zariski-closure of the image of $\varphi$, which is a projective variety in $\PP^{(n-1)^2}_{\R}$. Concretely, we view $Z_n$ as the projective closure of the set of $(n-1) \times (n-1)$ matrices which are the upper-left submatrix of an $n \times n$ orthostochastic matrix. In this way the linear equations which define (the linear span of) $\B_n$ are already accounted for in $Z_n$, and furthermore, the reduction of variables from $n^2$ to $(n-1)^2+1$ will be valuable for computation.

Introducing coordinates, we have that the map $\varphi$ above corresponds algebraically to the ring map
\[F : \R[y_{(1,1)}, \ldots, y_{(n-1,n-1)},s] \to \R[x_{1,1}, \ldots, x_{n,n}]/I_{\OO(n)} \]
\[y_{(i,j)} \mapsto x_{(i,j)}^2 \]
\[s \mapsto 1\]

\noindent Here $\R[x_{1,1}, \ldots, x_{n,n}]/I_{\OO(n)}$ is the affine coordinate ring of $\OO(n)$, and $\R[y_{(1,1)}, \ldots, y_{(n-1,n-1)},s]$ is the homogeneous coordinate ring of $\PP^{(n-1)^2}_\R$. Equations for $Z_n$ thus correspond to homogeneous forms in $\ker F$, and computing these will be our primary goal, since we can then obtain equations for the orthostochastic variety by dehomogenizing with respect to $s$, i.e. setting $s = 1$ (although one caveat arises with the hyperplane at infinity $s = 0$ -- this will be dealt with later).

We next give the dimension and degree of $Z_n$. Note that $\phi$ is a finite map, with general fibers of size $2^{n^2}$, corresponding to sign choices on each of the $n^2$ entries of a potential preimage. This implies that $\dim Z_n = \dim \OO(n)$, which we now recall: a matrix is orthogonal iff for all $i, j$, the dot product of the $i^\text{th}$ and $j^\text{th}$ rows equals $\delta_{ij}$, so
\[
I_{\OO(n)} = \langle \begin{bmatrix} x_{(i,1)} \ldots x_{(i,n)} \end{bmatrix}  \begin{bmatrix} x_{(j,1)} \ldots x_{(j,n)} \end{bmatrix}^T - \delta_{ij} \; \Big| \; 1 \le i \le j \le n \rangle
\]
In fact, these ${n \choose 2} + n = {n+1 \choose 2}$ quadrics form a regular sequence, so $\dim \OO(n) = n^2 - \codim I_{\OO(n)} = n^2 - {n+1 \choose 2} = {n \choose 2}$.

The degree of $Z_n$ is also known (albeit much more recently): the degree of $\SO(n)$ was computed in \cite{BBBKR17}, and in \cite{DGK18}, a general formula for the degree of a coordinate-wise power of a variety is given, and combining these yields:

\begin{proposition}[cf. \cite{BBBKR17}, Theorem 1.1, and \cite{DGK18}, Prop 2.4] \label{degFormula}
For any $n \ge 2$, 
\[\deg Z_n = 2^{(n-1)^2 - {n+1 \choose 2}} \deg \OO(n) = 2^{{n-1 \choose 2}} \det \begin{bmatrix} {2n - 2i - 2j \choose n - 2i} \end{bmatrix}_{i,j=1}^{\lfloor n/2 \rfloor}\]
\end{proposition}

\begin{remark}
1) Although $I_{\OO(n)}$ is a complete intersection of the ${n+1 \choose 2}$ quadric generators given above, the degree of $\OO(n)$ is not the product of the generator degrees, namely $2^{{n+1 \choose 2}}$, since the quadrics defining $\OO(n)$ are not homogeneous (it is true that the homogenizations of the quadrics define a homogeneous complete intersection of degree $2^{{n+1 \choose 2}}$, but that variety contains many more components besides those arising from $\OO(n)$).

2) Every orthostochastic matrix is in fact the image under $\phi$ of a special orthogonal matrix -- e.g. one may negate the first row without changing the coordinate-wise square. This shows that $Z_n$ is irreducible, being (the projective closure of) an image of an irreducible variety $\SO(n)$.
\end{remark}

For reference, we tabulate the values of $\dim Z_n, \deg Z_n$ for small values of $n$: \\

\begin{center}
\begin{tabular}{|c|c|c|c|c|}
\hline
    $n$ & 2 & 3 & 4 & 5 \\
\hline
    $\dim Z_n$ & 1 & 3 & 6 & 10 \\
\hline
    $\deg Z_n$ & 1 & 4 & 40 & 1536 \\
\hline
\end{tabular}
\end{center}

\section{$n = 3$ and naive equations} \label{section3}
We now review what is known about the defining equations of $Z_n$, for $n \le 3$. For $n = 2$, every doubly stochastic matrix is orthostochastic: indeed, a $2 \times 2$ doubly stochastic matrix is of the form $A = \begin{pmatrix} a & 1 - a \\ 1 - a & a \end{pmatrix}$ for some $0 \le a \le 1$, and writing $a =: \cos^2 \theta$ gives $1 - a = \sin^2 \theta$, so $A$ is the coordinate-wise square of the rotation matrix $\begin{pmatrix} \cos \theta & -\sin \theta \\ \sin \theta & \cos \theta \end{pmatrix} \in \SO(2)$. Thus no equations are needed to define $Z_2 = \PP^1$.

For $n = 3$, the variety $Z_3$ is a $3$-dimensional variety in $\PP^4$, hence is a hypersurface, and is defined by a single equation. We now show how to find this equation, first found in \cite{Nak96}, following the presentation in Section 3 of \cite{CD08}, which will in fact give a set of ``naive" equations for any $n$. Consider a $3 \times 3$ doubly stochastic matrix
\[
A = \begin{bmatrix}
y_{(1,1)} & y_{(1,2)} & y_{(1,3)} \\
y_{(2,1)} & y_{(2,2)} & y_{(2,3)} \\
y_{(3,1)} & y_{(3,2)} & y_{(3,3)}
\end{bmatrix}
\]
In order for $A$ to be orthostochastic, there must exist sign choices $\epsilon_1, \epsilon_2 \in \{ \pm 1 \}$ such that with these sign choices, the entrywise square roots of the first two columns are orthogonal:
\[
\sqrt{y_{(1,1)}} \sqrt{y_{(1,2)}} + \epsilon_1 \sqrt{y_{(2,1)}} \sqrt{y_{(2,2)}} + \epsilon_2 \sqrt{y_{(3,1)}} \sqrt{y_{(3,2)}} = 0
\]
By rearranging the above equation and squaring, we can eliminate one sign choice and all but one square root:
\[
y_{(1,1)} y_{(1,2)} + 2 \epsilon_1 \sqrt{y_{(1,1)}y_{(1,2)}y_{(2,1)}y_{(2,2)}} + y_{(2,1)}y_{(2,2)} = y_{(3,1)}y_{(3,2)}
\]
and another rearrangement and squaring produces a polynomial relation without $\epsilon_i$:
\[
(y_{(1,1)} y_{(1,2)} + y_{(2,1)}y_{(2,2)} - y_{(3,1)}y_{(3,2)})^2 = 4 y_{(1,1)}y_{(1,2)}y_{(2,1)}y_{(2,2)}
\]
and finally, using the fact that $A$ is doubly stochastic, we may express $y_{(3,1)}$ (resp. $y_{(3,2)}$) in terms of $y_{(1,1)}, y_{(2,1)}$ (resp. $y_{(1,2)}, y_{(2,2)}$) and homogenize with respect to $s$ to obtain
\[
(y_{(1,1)} y_{(1,2)} + y_{(2,1)}y_{(2,2)} - (s - y_{(1,1)} - y_{(2,1)})(s - y_{(1,2)} - y_{(2,2)}))^2 = 4 y_{(1,1)}y_{(1,2)}y_{(2,1)}y_{(2,2)}
\]
This defines a degree $4$ hypersurface in $\PP^4$ which contains $Z_3$, and therefore equals $Z_3$ by dimension and degree considerations.

In general, one can perform the same procedure for any pair of columns or rows of an $n \times n$ doubly stochastic matrix, to obtain a set of equations which any $n \times n$ orthostochastic matrix must satisfy:

\begin{definition}
For $1 \le i < j \le n$, let $C_{i,j}$ be the polynomial obtained by eliminating $\epsilon_1, \ldots, \epsilon_{n-1}, y_{(n,i)}, y_{(n,j)}$ from the relations 
\[
\sqrt{y_{(1,i)} y_{(1,j)}} + \sum_{k=2}^n \epsilon_{k-1} \sqrt{y_{(k,i)} y_{(k,j)}} = 0, \quad \epsilon_1^2 = \ldots = \epsilon_{n-1}^2 = 1, \quad y_{(n,i)} = 1 - \sum_{k=1}^{n-1} y_{(k,i)}, \; y_{(n,j)} = 1 - \sum_{k=1}^{n-1} y_{(k,j)}
\]
Note that $C_{i,j}$ is a polynomial of degree $2^{n-1}$ in $\R[y_{(1,1)}, \ldots, y_{(n-1,n-1)},s]$, obtained by repeatedly squaring (and rearranging) the first relation listed $n-1$ times (as in the case of $n = 3$ above), and that $C_{i,j}$ only involves variables in the $i^\text{th}$ and $j^\text{th}$ columns of a generic doubly stochastic matrix $(y_{(i,j)})$. Similarly, by transposing indices we define $R_{i,j}$ which only involves variables in the $i^\text{th}$ and $j^\text{th}$ rows, and refer to $\{ C_{i,j}, R_{i,j} \mid 1 \le i < j \le n \}$ as the set of \textit{naive equations}, of which there are $2 {n \choose 2} = n(n-1)$ equations in total.
\end{definition}

It follows from the discussion above that every $n \times n$ orthostochastic matrix must satisfy the naive equations $C_{i,j}, R_{i,j}$. More precisely, a doubly stochastic matrix $A$ satisfies $C_{i,j} = 0$ if and only if there exist choices of signs $\epsilon_1, \ldots, \epsilon_{n-1}$ such that \textit{with these choices of signs}, the entrywise square roots of the $i^\text{th}$ and $j^\text{th}$ columns of $A$ become orthogonal. A natural question is:

\begin{question}[cf. \cite{CD08}, Remark 3.4] \label{question-consistency}
Do the naive equations define the orthostochastic variety inside $\B_n$? In other words, if a doubly stochastic matrix satisfies $C_{i,j}, R_{i,j} = 0$ for all $i < j$, must it be orthostochastic?
\end{question}

Stated another way: satisfying a single equation $C_{i,j}$ only guarantees sign choices which work for a single pair of columns. \Cref{question-consistency} asks whether existence of \textit{local} sign choices for each pair of columns (and rows), implies existence of a single \textit{global} sign choice which simultaneously makes all pairs of columns orthogonal.

As an example of what could go wrong, it is conceivable that (after fixing sign choices on the first two columns as required by $C_{1,2}$) the sign choices imposed on the third column by $C_{1,3}$ could differ from those imposed by $C_{2,3}$. This turns out not to happen in the case of $n = 3$, as $Z_3$ is defined by the vanishing of (the homogenization of) $C_{1,2}$. However, as noted in \cite{CD08}, this good behavior is indeed special only to small values of $n$, as evidenced by the following proposition (which justifies our use of the term naive):

\begin{proposition} \label{prop-ge6}
For $n \ge 6$, there exists a doubly stochastic matrix which satisfies $C_{i,j}, R_{i,j} = 0$ for all $i < j$, but is not orthostochastic. Thus \Cref{question-consistency} has a negative answer for $n \ge 6$.
\end{proposition}

\begin{proof}
Let $A := (\frac{1}{6}J_6) \oplus I_{n-6}$ be block diagonal, where $J_6$ is the $6 \times 6$ matrix of all $1$'s, and $I_{n-6}$ is the identity matrix of size $n - 6$. Then $A$ is doubly stochastic, and it is straightforward to check that $A$ satisfies all $C_{i,j}, R_{i,j}$: since $A$ is symmetric, it suffices to check that $R_{i,j}$ is satisfied, i.e. there exist local sign choices which make the entrywise square roots of the $i^\text{th}$ and $j^\text{th}$ rows orthogonal. If $j > 6$ then the $i^\text{th}$ and $j^\text{th}$ rows are already orthogonal, and if $i < j \le 6$ then choosing all positive signs for the $i^\text{th}$ row and half positive, half negative for the $j^\text{th}$ row shows that $R_{i,j}$ is satisfied. On the other hand, $\frac{1}{6} J_6$ is not orthostochastic, since there does not exist a $6 \times 6$ Hadamard matrix, and therefore $A$ is not orthostochastic, since a direct sum of square matrices is orthostochastic iff each matrix is orthostochastic.
\end{proof}

\begin{remark}
It was noted in \cite{CD08} that if $n \ge 6$ satisfies $n \equiv 2 \pmod 4$, then $J_n$ satisfies $C_{i,j}, R_{i,j} = 0$ but is not orthostochastic, and it was also proven that counterexamples exist for all $n \ge 16$, using Hurwitz-Radon theory. However, the simple explicit argument with direct sums given in \Cref{prop-ge6} seems to have gone unnoticed.
\end{remark}

\section{$n = 4$} \label{section4}
We now arrive at the main focus of this paper, the case $n = 4$. In this section, most of the methods and arguments we use will be primarily computational in nature, so we first remark on the techniques used. 

Computations were performed in Macaulay2 \cite{GS} and Bertini \cite{BHSW06}, and involve a mix of symbolic and numerical methods. All symbolic computations were done in exact arithmetic, using Gr\"{o}bner bases over the rational numbers -- this is possible since all the varieties in question are defined over $\Q$. For the purposes of this article, we regard results obtained by symbolic computation to be on par with those obtained by theoretical proof.

On the other hand, numerical computations were done in floating point, to 53 bits of precision (i.e. standard double-precision floating-point). Results obtained by numerical methods are regarded as correct ``with probability 1": the fact that these programs give reproducible results agreeing with theory in thousands of cases allows us to say with overwhelming confidence that the results obtained in this way are correct. Throughout the section we indicate when a computation was used, as well as the type (symbolic or numerical). We encourage the interested reader to confirm the results of the computations themselves, included in the appendix/auxiliary files.

Returning to the variety at hand: $Z_4$ is a $6$-dimensional variety in $\PP^9$ of degree $40$ (we remark that by coincidence, this is the unique value for $n \ge 2$ for which $\deg Z_n = \deg \SO(n)$). As before, we have the naive equations $C_{i,j}, R_{i,j}$ for $1 \le i < j \le 4$, which constitute $12$ octics whose vanishing locus (after taking projective closure) contains $Z_4$.

\begin{definition}
We set $Y_c := V(\widetilde{C_{i,j}} : 1 \le i < j \le 4)$ and $Y_r := V(\widetilde{R_{i,j}} : 1 \le i < j \le 4)$ as the varieties defined by the homogenizations of $6$ naive octics corresponding to (all pairs of) columns and rows, respectively. We also define $Y := Y_c \cap Y_r$. Note that these are all subvarieties of the same ambient space $\PP^9$ as $Z_4$.
\end{definition}

We know that $Z_4 \subseteq Y$, and in fact the two agree generically:

\begin{proposition} \label{propY}
If $L \subseteq \PP^9$ is a general linear space of codimension $6$, then $Y \cap L$ consists of $40$ points.
\end{proposition}

\begin{proof}
Since the claim is for a general linear space $L$, we may show this by direct computation. Choosing a linear slice $L$ at random (i.e. $6$ linear forms in $10$ variables with random rational coefficients), upon computing a minimal presentation of the ideal of $Y \cap L$ we arrive at an ideal generated by $12$ octics in $4$ variables, whose dimension and degree can easily be computed symbolically to be 0 and 40, respectively.
\end{proof}

\Cref{propY} shows that $Y$ has the same dimension and degree as $Z_4$. In particular, $Z_4$ (being irreducible) is the unique top-dimensional component of $Y$, and so $Y \ne Z_4$ if and only if $Y$ contains additional components of lower dimension. Thus, if $I(Y) = (C_{i,j}, R_{i,j})$ denotes the ideal generated by the 12 naive octics, then showing that $I(Y)$ has only one minimal prime (or equivalently, that the radical of $I(Y)$ is prime) would imply $Y = Z_4$.

However, the basic issue is that $I(Y)$ is too large to handle directly: half of the octics have 967 terms, and the other half have 6760 terms. Neither symbolic nor numerical methods (via numerical irreducible decomposition) were able to determine the number of minimal primes of $I(Y)$. This motivates our search for other, ``smaller" equations of $Z_4$.

How does one find equations for a parametrized variety? This procedure is known as implicitization, and amounts to computing the kernel of a ring map. Classically, this is accomplished with Gr\"{o}bner bases, but in the $n = 4$ case, this is infeasible (as of yet). Thus, we try a numerical approach, using \textit{interpolation} (cf. Theorem 3 in \cite{CK18}). The method is as follows: since $Z_4$ is a projective variety, its defining ideal $I(Z_4)$ has homogeous generators, so we may search degree by degree. Fixing a degree $d$, we may find a basis of $I(Z_4)_d$ (the degree $d$ part of $I(Z_4)$) by sampling a large number of points on $Z_4$, evaluating all monomials of degree $d$ in $\PP^9$ at these points, and then computing the numerical kernel of the resulting matrix (with rows indexed by points, and columns indexed by monomials). The kernel vectors are then coefficients of degree $d$ forms (linearly independent over $\R$) which vanish at all of the sampled points, and if the number of points is large, one expects such a form to vanish on all of $Z_4$. There are no linear forms in $I(Z_4)$ (as these are already accounted for by restricting to $\B_4$), and similarly we find no forms of degrees $2, 3$, and $4$. However, in degree $5$ we find:

\begin{proposition} \label{prop-quintics}
The space of quintics in $I(Z_4)$ is $6$-dimensional, with explicitly known basis.
\end{proposition}

\begin{proof}
Note that sampling points on $Z_4$ is easily accomplished: one can generate a random $4 \times 4$ orthogonal matrix, e.g. via the Cayley parameterization of $\SO(n)$, and then apply the map $\varphi$ to obtain a point on $Z_4$. In total, if $A$ is a $4 \times 4$ matrix with entries chosen at random, then the final output of $A \mapsto B := \frac{1}{2}(A - A^T) \mapsto C := \frac{I-B}{I+B} \mapsto \varphi(C)$ is a (random) point on $Z_4$.

With this, determining numerically the dimension of $I(Z_4)_5$ by the method outlined above is straightforward and relatively quick, using the Macaulay2 package \texttt{NumericalImplicitization}, which returns an answer of $6$ within $1$ minute. However, this is only a numerical computation, which provides approximate quintics with floating-point coefficients.

To obtain explicit quintics over $\Z$ from the approximate quintics, we use the LLL algorithm, again implemented in \texttt{NumericalImplicitization} (which in turn calls the native LLL implementation in Macaulay2). This calculation is significantly longer, taking around $30$ hours to finish, but results in a $2002 \times 6$ integer matrix, whose columns are coefficients of degree 5 forms in $\PP^9$. Once these quintics are known, they are easily checked to be in $I(Z_4)$: verifying symbolically that the integer quintics lie in $\ker F$ takes $< 1$ second.
\end{proof}

\begin{definition} \label{def:quintics}
We set $J$ to be the ideal generated by the $6$ quintics found in \Cref{prop-quintics}, and let $X := V(J)$ be the variety they define.
\end{definition}

Although still too large to comfortably reproduce here, the quintics in $J$ are significantly more manageable than the octics $C_{i,j}, R_{i,j}$: 4 of the quintics have 284 terms, and the most complicated involves 454 terms. In fact a Gr\"obner basis of $J$ can be computed (taking around 8 hours), which yields $\dim X = 6$, $\deg X = 40$ (the same as $Z_4$). Since we also know $Z_4 \subseteq X$ as well, we may again test whether equality holds by computing the number of minimal primes of $J$, i.e. irreducible components of $X$. Although symbolic computation of minimal primes of $J$ is still too slow to be practical, the key difference from the previous case with $Y$ is that $X$ is small enough for a numerical irreducible decomposition to be feasible. The result, though initially negative, is immediately promising and leads to a resolution of the problem.

\begin{proposition} \label{prop-NID}
The numerical irreducible decomposition of $X$ consists of:
\begin{enumerate}
\item $16$ components of dimension $4$, each of degree $1$,
\item $15$ components of dimension $5$, each of degree $1$, 
\item $1$ component of dimension $6$, of degree $40$.
\end{enumerate}
\end{proposition}

\begin{proof}
Using Bertini \cite{BHSW06}, we may compute a numerical irreducible decomposition of $J$, which finishes within $30$ minutes. Note that the result is consistent with the Gr\"obner basis of $J$ computed symbolically, which implies that $Z_4$ is the unique top-dimensional component of $X$.
\end{proof}

The data of a numerical irreducible decomposition of a variety (given by defining equations) consists of a collection of witness sets, each representing an irreducible component of the variety, along with linear equations defining a complementary-dimensional slice of each witness set, and the finitely many intersection points of each witness set with its linear slice, thus encoding the dimension and degree of each irreducible component. Since all the lower-dimensional components of $X$ obtained in the numerical irreducible decomposition were themselves linear spaces (being degree 1), one should expect that equations for them can be found and moreover interpreted as certain classes of $4 \times 4$ matrices. This indeed turns out to be the case:

\begin{enumerate}
\item The 4-dimensional components correspond to the 16 ways of specifying that one entry of a $4 \times 4$ doubly stochastic matrix be equal to $1$: this is a codimension $5$ constraint in $\B_4$, as it forces the remaining entries in that row and column to be $0$. For instance, one such component is defined by the ideal $(y_{(2,1)}, y_{(2,2)}, y_{(1,3)}, y_{(3,3)}, y_{(2,3)} - s)$, which (after dehomogenizing $s = 1$) requires that the $(2,3)$-entry is $1$. Up to row and column permutations, such matrices are a direct sum of a $3 \times 3$ matrix with the $1 \times 1$ identity.

\item The 5-dimensional components are slightly more complicated: of these, 6 are contained in the hyperplane at infinity $s = 0$, and thus are irrelevant for the orthostochastic variety. The 9 \textit{relevant} components are as follows: fix a pair of indices $(a,a')$ with $a, a' \in \{2,3,4\}$, and write $\{2,3,4\} \setminus \{a\} =: \{b, c\}, \{2,3,4\} \setminus \{a'\} =: \{b', c'\}$. Then the ideal $(y_{(1,1)} - y_{(a,a')}, y_{(a,1)} - y_{(1,a')}, y_{(b,1)} - y_{(c,a')}, y_{(1,b')} - y_{(a,c')})$ defines a 5-dimensional component of $X$. Note that these components can be described by partitioning a $4 \times 4$ doubly stochastic matrix into 4 disjoint $2 \times 2$ submatrices, and requiring each submatrix to be circulant (i.e. has equal cross terms).
\end{enumerate}

These results were obtained by following the same procedure as in the proof of \Cref{prop-quintics} (sampling points, obtaining approximate linear equations, and then extracting exact linear equations via LLL), and the resulting equations are symbolically verified to define subvarieties of $X$. In the case of the 5-dimensional components, an additional, carefully chosen, change-of-basis was necessary to obtain binomial linear generators, and with it the ensuing description as matrices. The explicit linear ideals are recorded (as \texttt{o22} and \texttt{o24}) in the appendix below.

Each of the relevant lower-dimensional components of $X$ can thus be defined over $\Z$ and interpreted as a particular class of matrices. With these descriptions via integer equations at hand, it is also straightforward to confirm that each of these classes of matrices defines an actual component of $X$, via checking dimensions of tangent spaces. Finally, we may use this to obtain \Cref{mainThm}:

\begin{proof}[Proof of \Cref{mainThm}]
Let $K := (C_{1,2},C_{1,3},C_{2,3})$ be the ideal generated by 3 of the octics, corresponding to (pairs of) the first 3 columns. We claim that $J$ and $K$ together cut out the $4 \times 4$ orthostochastic variety set-theoretically, i.e. $(Z_4)_{s=1} = (X \cap V(K))_{s=1}$ as sets (here $(\underline{\hspace{0.2cm}})_{s=1}$ means dehomogenize with respect to $s$). In view of the numerical irreducible decomposition of $X$ computed in \Cref{prop-NID}, it suffices to show that for each of the relevant lower-dimensional components $C$ of $X$, we have $(C \cap Z_4)_{s=1} = (C \cap V(K))_{s=1}$.

For $Q$ a $4$-dimensional component of $X$, we have that a point $p \in (Q \cap Z_4)_{s=1}$ is (up to row and column permutation) a direct sum of a $1 \times 1$ and a $3 \times 3$ orthostochastic matrix. The case $n = 3$ in \Cref{section3} implies that $Q \cap Z_4$ is defined by a single form of degree $4$, and we check symbolically that $Q \cap V(K)$ is in fact defined by a single octic, which is precisely the square of the quartic defining $Q \cap Z_4$.

For $P$ a relevant $5$-dimensional component of $X$, we again check that $P \cap V(K)$ is defined by a single octic. In this case we do not have equations for $P \cap Z_4$ a priori -- however, the structure of the matrices in $P$ is special enough to make symbolic implicitization feasible. Indeed, we may compute symbolically the kernel of the ring map
\[
F \!\!\!\! \pmod {(I(P)+(s-1))} : \R[y_{(1,1)}, \ldots, y_{(3,3)}]/I(P) \to \R[x_{(1,1)}, \ldots, x_{(4,4)}]/(I_{\OO(4)} + F(I(P)))
\]
which takes roughly 8 minutes -- this is made possible by the fact that after minimizing the presentation, the source is really a polynomial ring in 5 variables. This gives equations for $(P \cap Z_4)_{s=1}$, which turns out to be precisely the dehomogenization of the octic defining $P \cap V(K)$.

Putting this all together, we have that if $l_1, \ldots, l_{16}$ are the $4$-dimensional components of $X$, and $L_1, \ldots, L_{15}$ are the $5$-dimensional components of $X$, where $L_{10}, \ldots, L_{15}$ are contained in the hyperplane $s = 0$, then
\begin{align*}
(X \cap V(K))_{s=1} &= \left( \Big( Z_4 \cup \bigcup_{i=1}^{16} l_i \cup \bigcup_{i=1}^{15} L_i \Big) \cap V(K) \right)_{s=1} \\
&= (Z_4 \cap V(K))_{s=1} \cup \bigcup_{i=1}^{16} (l_i \cap V(K))_{s=1} \cup \bigcup_{j=1}^9 (L_i \cap V(K))_{s=1} \\
&= (Z_4)_{s=1} \cup \bigcup_{i=1}^{16} (l_i \cap Z_4)_{s=1} \cup \bigcup_{j=1}^9 (L_i \cap Z_4)_{s=1} = (Z_4)_{s=1}  \qedhere
\end{align*}
\end{proof}

\section{Conclusion}
A few remarks on the methodology used in \Cref{section4} are in order. First, the only part of the proof of Theorem which relies on numerical computation (i.e. for which there is no symbolic verification) is the correctness of the numerical irreducible decomposition \Cref{prop-NID}. Specifically, what is required is that there are no other irreducible components of $X$ other than those listed in \Cref{prop-NID}. For readers concerned about the random choices involved in path tracking and homotopy continuation, we mention that the numerical irreducible decomposition of $X$ has been calculated multiple times, each time giving the same consistent result. Moreover, the numerical irreducible decomposition of the affine variety $X_{s=1}$ (the dehomogenization of $X$) has also been computed, and is consistent with \Cref{prop-NID} (e.g. among the $5$-dimensional components, only the $9$ relevant components not at infinity are present).

One may also ask about the utility of set-theoretic equations for $Z_4$. After all, given a particular $4 \times 4$ matrix, the brute-force check to determine if it is orthostochastic is still feasible (namely checking all possible sign patterns in a fiber of $\phi$. Since $(\Z/2\Z)^n \oplus (\Z/2\Z)^n$ acts on fibers of $\phi$ by row and column sign changes, one may assume the first row and column are all nonnegative, so there are only $2^{(4-1)^2} = 512$ sign patterns to check). In response to this, we note that in addition to their intrinsic interest, knowledge of set-theoretic equations is extremely useful for many other purposes. For instance, given a variety which meets the set of $4 \times 4$ orthostochastic matrices, one can now describe their common intersection locus (which was previously not possible to do). We plan to use this in a forthcoming upgrade to \cite{CD19}, to compute monic symmetric determinantal representations of hyperbolic plane quartic curves.

As for the equations themselves, it is natural to ask about minimality. For degree reasons, any set of minimal generators of $I(Z_4)$ must include the $6$ quintics $J$ in \Cref{def:quintics}. It has been checked that the spaces of sextics and 7-forms in $I(Z_4)$ are $60$- and $330$-dimensional respectively, and moreover that the quintics in $J$ have no quadratic (or linear) syzygies. It follows that there are no forms of degree $6$ or $7$ in $I(Z_4)$ other than those already in $J$. For degree $8$, it is not difficult to see that $2$ of the naive octics (along with $J$) would not be sufficient to cut out $Z_4$, so the generating set in \Cref{mainThm} is both inclusion-wise and degree-wise minimal.

Finally, we list some questions that remain unanswered with this work. First, it would be interesting to have a combinatorial description of the $6$ quintics $J$, as well as a theoretical proof for why they vanish on $4 \times 4$ orthostochastic matrices. Next, for $n = 4, 5$ it remains open whether the naive equations cut out the orthostochastic variety inside the Birkhoff polytope set-theoretically. 
Lastly, we know that for $n \ge 6$ additional equations are needed beyond the naive ones. However, finding where these come from is related to the first question listed here, and even a computational proof of sufficiency may only be possible with some advances in computing technology.


\section{Acknowledgements}
The first author would like to thank Anton Leykin for helpful discussions, especially related to verifying components via tangent spaces. The second author would like to thank Paul Breiding, Mateusz Michalek and Bernd Sturmfels for their encouragement towards work on this project.

\appendix
\section{M2 session}
The following is an example Macaulay2 session that demonstrates the main results in this paper. For technical reasons, it is more convenient to use $y_0, \ldots, y_9$ as the variables of $\PP^9$ rather than $y_{(1,1)}, \ldots, y_{(3,3)}, s$.

\begin{verbatim}
Macaulay2, version 1.15
i1 : needsPackage "DeterminantalRepresentations"
i2 : needsPackage "NumericalImplicitization"
i3 : S = RR[x_(1,1)..x_(4,4)]; R = QQ[y_0..y_9]; s = y_9;
i6 : A = genericMatrix(S,4,4); IO4 = minors(1, A*transpose A - id_(S^4));
i8 : F = matrix{flatten entries submatrix(hadamard(A,A),{0,1,2},{0,1,2})} | matrix{{1_S}}
-- Sample random points on Z_4
i9 : pts = apply(binomial(9+5,5), i -> sub(F,matrix{flatten entries randomOrthogonal(4,RR)}));
-- Compute dimension of deg 5 part of I(Z_4): takes ~30 seconds
i10 : HF5 = numericalHilbertFunction(F, IO4, pts, 5, UseSLP => true, Precondition => false)
-- Get integer quintics via LLL: takes ~30 hours
i11 : time E = extractImageEquations(HF5, AttemptZZ => true);
-- Alternatively, after downloading the auxiliary files, one can load the equations with:
-- i11 : E = first lines get "ortho4-quintics.txt";
i12 : J = ideal value toString E;
-- Create the s-homogenization of a generic 4x4 doubly stochastic matrix
i13 : V = genericMatrix(R,3,3); V = transpose(V || matrix{toList(3:s) - sum entries V})
i15 : V = V || matrix{toList(4:s) - sum entries V}
-- 3 naive octics C_{1,2}, C_{1,3}, C_{2,3}
i16 : K = ideal apply(subsets(3, 2), p -> ( 
        (i,j) = (p#0, p#1);
        ((V_(i,0)*V_(j,0) + V_(i,1)*V_(j,1) - V_(i,2)*V_(j,2) -
        V_(i,3)*V_(j,3))^2 - 4*V_(i,0)*V_(j,0)*V_(i,1)*V_(j,1) -
        4*V_(i,2)*V_(j,2)*V_(i,3)*V_(j,3))^2 -
        64*V_(i,0)*V_(j,0)*V_(i,1)*V_(j,1)*V_(i,2)*V_(j,2)*V_(i,3)*V_(j,3)
      )); -- cf. equation (3.5) in [7], esp. the coefficient of 64
i17 : time numericalIrreducibleDecomposition(J, Software => BERTINI) -- takes ~25 minutes
-- Note: the dimensions shown here are of the affine cones, 
-- which are 1 more than that of the corresponding projective variety
i18 : (X, eps) = (oo, 1e-10);
-- These define some helper functions to obtain descriptions of linear components
i19 : realPartMatrix := A -> matrix apply(entries A, r -> r/realPart)
i20 : imPartMatrix := A -> matrix apply(entries A, r -> r/imaginaryPart)
i21 : getLinEqs = (A, mons, c, n) -> (
        B = random(RR)*realPartMatrix A + random(RR)*imPartMatrix A;
        C = matrix apply(entries B, r -> r/(e -> lift(round(10^(1+n)*round(n, e)), ZZ)));
        mons*submatrix(LLL(id_(ZZ^(numcols C))||C), toList(0..<numcols mons), toList(0..<c))
      )
-- The following are the 4-dim (= codim 5) components of X, recorded in o22
i22 : L5 = apply(X#5, W -> ideal getLinEqs(clean(eps, matrix W#Points#0), basis(1,R), 5, 10))
o22 = {ideal(y_0, y_3, y_6, -y_2-y_5-y_8+y_9, -y_1-y_4-y_7+y_9),
      ideal(y_0, y_2, y_4, y_7, -y_1+y_9),
      ideal(y_2, y_5, y_6, y_7, -y_8+y_9),
      ideal(y_2, y_3, y_4, y_8, -y_5+y_9),
      ideal(y_1, y_2, y_3, y_6, -y_0+y_9),
      ideal(y_0, y_1, y_2, -y_6-y_7-y_8+y_9, -y_3-y_4-y_5+y_9),
      ideal(y_0+y_3+y_6-y_9,y_0+y_1+y_2-y_9,y_6+y_7+y_8-y_9,y_1+y_4+y_7-y_9,y_2+y_5+y_8-y_9),
      ideal(y_0, y_3, y_7, y_8, -y_6+y_9),
      ideal(y_6, y_7, y_8, -y_3-y_4-y_5+y_9, -y_0-y_1-y_2+y_9),
      ideal(y_1, y_3, y_5, y_7, -y_4+y_9),
      ideal(y_0, y_4, y_5, y_6, -y_3+y_9),
      ideal(y_3, y_4, y_5, -y_6-y_7-y_8+y_9, -y_0-y_1-y_2+y_9),
      ideal(y_0, y_1, y_5, y_8, -y_2+y_9),
      ideal(y_1, y_4, y_7, -y_0-y_3-y_6+y_9, y_2+y_5+y_8-y_9),
      ideal(y_2, y_5, y_8, -y_0-y_3-y_6+y_9, y_1+y_4+y_7-y_9),
      ideal(y_1, y_4, y_6, y_8, -y_7+y_9)}
-- Check: description as direct sums of 3x3 ++ 1x1 matrices, Q \cap Z_4 = Q \cap V(K)
i23 : all((0,0)..(3,3), p -> (
      (i,j) = p; 
      L = ideal(V_(i,j) - s, V_((i+1)%4,j), V_((i+2)%4,j), V_(i,(j+1)%4), V_(i,(j+2)%4));
      any(L5, l -> l == L) and isSubset(J, L) and (
        Q = minimalPresentation sub(K, R/L);
        T = ring Q;
        f = (T_4^2 - T_4*(T_0+T_1+T_2+T_3) + (T_0*T_3+T_1*T_2))^2 - 4*T_0*T_1*T_2*T_3;
        ideal(f^2) == Q
      )))
-- The following are the 5-dim (= codim 4) components of X, recorded in o24
i24 : L4 = apply(X#6, W -> ideal getLinEqs(clean(eps, matrix W#Points#0), basis(1,R), 4, 10))
o24 = {ideal(-y_1+y_6, -y_0+y_7, -y_1-y_3-y_4-y_7+y_9, y_2-y_3-y_4+y_8),
      ideal(-y_2+y_6, -y_0+y_8, y_1-y_3-y_5+y_7, -y_1-y_2-y_7-y_8+y_9),
      ideal(-y_1+y_3, -y_0+y_4, -y_2-y_5+y_6+y_7, -y_2-y_3-y_4-y_5+y_9),
      ideal(y_9, y_0+y_1, y_3+y_4, y_6+y_7),
      ideal(-y_4+y_6, -y_3+y_7, -y_0-y_1+y_5+y_8, -y_3-y_5-y_6-y_8+y_9),
      ideal(y_9, y_0+y_6, y_1+y_7, y_2+y_8),
      ideal(-y_5+y_6, -y_3+y_8, -y_0-y_2-y_5-y_8+y_9, -y_0-y_2+y_4+y_7),
      ideal(-y_2+y_4, -y_1+y_5, -y_0-y_3+y_7+y_8, -y_1-y_2-y_7-y_8+y_9),
      ideal(y_9, y_3+y_6, y_4+y_7, y_5+y_8),
      ideal(-y_5+y_7, -y_4+y_8, -y_1-y_2+y_3+y_6, -y_1-y_2-y_5-y_8+y_9),
      ideal(-y_2+y_3, -y_0+y_5, -y_1-y_4+y_6+y_8, -y_2-y_5-y_6-y_8+y_9),
      ideal(y_9, y_1+y_2, y_4+y_5, y_7+y_8),
      ideal(-y_2+y_7, -y_1+y_8, y_0-y_4-y_5+y_6, -y_1-y_2-y_4-y_5+y_9),
      ideal(y_9, y_0+y_2, y_3+y_5, y_6+y_8),
      ideal(y_9, y_0+y_3, y_1+y_4, y_2+y_5)}
-- Remove 6 components at infinity
i25 : L4 = select(L4, l -> s % l != 0); #L4
o26 = 9
-- Check: description as block 2x2 circulant matrices, and containment in X
i27 : all((1,1)..(3,3), p -> ( 
        (a,a') = p; 
        (b,c) = toSequence({1,2,3} - set{a});
        (b',c') = toSequence({1,2,3} - set{a'});
        L = ideal(V_(0,0)-V_(a,a'),V_(a,0)-V_(0,a'),V_(b,0)-V_(c,a'),V_(0,b')-V_(a,c'));
        any(L4, l -> l == L) and isSubset(J, L)
      ))
-- Work over QQ for exact implicitization
i28 : S = QQ[x_(1,1)..x_(4,4)]; A = genericMatrix(S,4,4);
i30 : IO4 = minors(1, A*transpose A - id_(S^4));
i31 : H = matrix{flatten entries submatrix(hadamard(A,A),{0,1,2},{0,1,2})}; 
i32 : F = H | matrix{{1_S}}
-- Check: (P \cap Z_4)_{y_9=1} = (P \cap V(K))_{y_9=1}
i33 : all(L4, I -> (
        P = minimalPresentation sub(K + ideal(s - 1), R/I);
        time ker map(S/(IO4 + sub(I, F)), ring P, H_((gens ring P)/baseName/last)) == P
      )) -- takes ~9*8 minutes
-- Check: J is contained in ker F
i34 : sub(gens J, F) % IO4 == 0
-- Check: J has no quadratic syzygies
i35 : syz(gens J, DegreeLimit => 7) == 0
-- To compute the dimension of the space of 7-forms in I(Z_4): 
-- repeat lines i3 - i10, replacing everywhere 5 with 7 (takes ~7 hours, ~34 GB RAM)
\end{verbatim}
\end{document}